\newfont{\bcb}{msbm10}
\newfont{\matb}{cmbx10}
\newfont{\got}{eufm10}

\documentclass[12pt]{amsart}
\usepackage{amsmath, amsthm, amscd, amsfonts, amssymb, latexsym, graphicx, color}
\usepackage[bookmarksnumbered, colorlinks, plainpages, hypertex]{hyperref}

\usepackage[cp1250]{inputenc}

\usepackage{amsmath,amsthm}
\usepackage{amssymb,latexsym}
\usepackage{enumerate}

\newtheorem{theorem}{Theorem}%[section]

\newtheorem{proposition}[theorem]{Proposition}
\newtheorem{corollary}[theorem]{Corollary}
\theoremstyle{definition}

\theoremstyle{remark}
\newtheorem{remark}[theorem]{Remark}
\numberwithin{equation}{section}

\begin{document}

\title[Density property]{A density property \\ of Henselian valued fields}

\author[Krzysztof Jan Nowak]{Krzysztof Jan Nowak}

%\footnotetext{Research partially supported by by KBN Grant No.\
%1P03A 00527.}



%\date{}

%\vspace{3ex}

\begin{abstract}
We give an elementary proof of a version of the implicit function
theorem over Henselian valued fields $K$. It yields a density
property for such fields (introduced in a joint paper with
J.~Koll{\'a}r), which is indispensable for ensuring reasonable
topological and geometric properties of algebraic subsets of
$K^{n}$.
\end{abstract}

\maketitle

\vspace{1ex}

Following~\cite{K-N} (see also~\cite{Now2}), we say that a
topological field $K$ satisfies the {\it density property} if the
following equivalent conditions hold.
\begin{enumerate}
\item If $X$ is a smooth, irreducible $K$-variety and
    $\emptyset\neq U\subset X$ is a Zariski open subset,
then $U(K)$ is dense in $X(K)$ in the $K$-topology.
\item If $C$ is a smooth, irreducible $K$-curve and
    $\emptyset\neq U$ is a Zariski open subset, then $U(K)$ is
    dense in
$C(K)$ in the $K$-topology.
\item If $C$ is a smooth, irreducible $K$-curve, then $C(K)$
    has no isolated points.
\end{enumerate}
This property is indispensable for ensuring reasonable topological
and geometric properties of algebraic subsets of $K^{n}$;
see~\cite{Now2} for the case where the ground field $K$ is a
Henselian rank one valued field. For Henselian non-trivially
valued fields, it can be directly deduced from the Jacobian
criterion for smoothness and the implicit function theorem, as
stated in \cite[Theorem~7.4]{P-Z} or
\cite[Proposition~3.1.4]{G-G-MB}. Here we give elementary proofs
of some versions of the inverse mapping and implicit function
theorems.

We begin with a simplest version of Hensel's lemma in several
variables, studied by Fisher~\cite{Fish}. Let
$\mathfrak{m}^{\times n}$ stand for the $n$-fold Cartesian product
of $\mathfrak{m}$ and $R^{\times}$ for the set of units of $R$.
The origin $(0,\ldots,0) \in R^{n}$ is denoted by $\mathbf{0}$.

\vspace{1ex}

\begin{em}
{\bf (H)} Assume that a ring $R$ satisfies Hensel's conditions
(i.e.\  it is linearly topologized, Hausdorff and complete) and
that an ideal $\mathfrak{m}$ of $R$ is closed. Let $f = (f_{1},
\ldots, f_{n})$ be an $n$-tuple of restricted power series $f_{1},
\ldots, f_{n} \in R\{ X \}$, $X = (X_{1},\ldots,X_{n})$, $J$ be
its Jacobian determinant and $a \in R^{n}$. If $f(\mathbf{0}) \in
\mathfrak{m}^{\times n}$ and $J(\mathbf{0}) \in R^{\times}$, then
there is a unique $a \in \mathfrak{m}^{\times n}$ such that $f(a)
= \mathbf{0}$.
\end{em}

\begin{proposition}\label{H-1}
Under the above assumptions, $f$ induces a bijection
$$ \mathfrak{m}^{\times n} \ni x \to f(x) \in \mathfrak{m}^{\times n} $$
 of $\mathfrak{m}^{\times n}$ onto itself.
\end{proposition}

\begin{proof}
For any $y \in\mathfrak{m}^{\times n}$, apply condition (H) to the
restricted power series $f(X) - y$.
\end{proof}

If, moreover, the pair $(R,\mathfrak{m})$ satisfies Hensel's
conditions (i.e.\ every element of $\mathfrak{m}$ is topologically
nilpotent), then condition (H) holds by \cite[Chap.~III, \S
4.5]{Bour}.

\begin{remark}\label{rem-char}
Henselian local rings can be characterized both by the classical
Hensel lemma and by condition (H): a local ring $(R,\mathfrak{m})$
is Henselian iff $(R,\mathfrak{m})$ with the discrete topology
satisfies condition (H) (cf.~~\cite[Proposition~2]{Fish}).
\end{remark}

Now consider a Henselian local ring $(R,\mathfrak{m})$. Let $f =
(f_{1}, \ldots, f_{n})$ be an $n$-tuple of polynomials $f_{1},
\ldots, f_{n} \in R[ X ]$, $X = (X_{1},\ldots,X_{n})$ and $J$ be
its Jacobian determinant.

\begin{corollary}\label{H-2}
Suppose that $f(\mathbf{0}) \in \mathfrak{m}^{\times n}$ and
$J(\mathbf{0}) \in R^{\times}$. Then $f$ is a homeomorphism of
$\mathfrak{m}^{\times n}$ onto itself in the $\mathfrak{m}$-adic
topology. If, in addition, $R$ is a Henselian valued ring with
maximal ideal $\mathfrak{m}$, then $f$ is a homeomorphism of
$\mathfrak{m}^{\times n}$ onto itself in the valuation topology.
\end{corollary}

\begin{proof}
Obviously, $J(a) \in R^{\times}$ for every $a \in
\mathfrak{m}^{\times n}$. Let $\mathcal{M}$ be the jacobian matrix
of $f$. Then
$$ f(a + x) - f(a) = \mathcal{M}(a) \cdot x + g(x) = \mathcal{M}(a)
   \cdot (x + \mathcal{M}(a)^{-1} \cdot g(x)) $$
for an $n$-tuple $g = (g_{1},\ldots,g_{n})$ of polynomials
$g_{1},\ldots,g_{n} \in (X)^{2} R[X]$. Hence the assertion follows
easily.
\end{proof}

The proposition below is a version of the inverse mapping theorem.

\begin{proposition}\label{H-3}
If $f(\mathbf{0}) = \mathbf{0}$ and $e :=J(\mathbf{0}) \neq 0$,
then $f$ is an open embedding of $e^{2} \cdot \mathfrak{m}^{\times
n}$ into $e \cdot \mathfrak{m}^{\times n}$.
\end{proposition}

\begin{proof}
Let $\mathcal{N}$ be the adjugate of the matrix
$\mathcal{M}(\mathbf{0})$ and $y = e^{2}b$ with $b \in
\mathfrak{m}^{\times n}$. Since
$$ f(X) = e \mathcal{M}(a) \cdot X + e^{2} g(x) $$
for an $n$-tuple $g = (g_{1},\ldots,g_{n})$ of polynomials
$g_{1},\ldots,g_{n} \in (X)^{2} R[X]$, we get the equivalences
$$ f(eX) = y \ \Leftrightarrow \ f(eX) - y = \mathbf{0} \ \Leftrightarrow \
   e \mathcal{M}(\mathbf{0}) \cdot (X + \mathcal{N}g(X) - \mathcal{N}b) = \mathbf{0}. $$
Applying Corollary~\ref{H-2} to the map $h(X) := X +
\mathcal{N}g(X)$, we get
$$ f^{-1}(y) = ex \ \Leftrightarrow \ x = h^{-1}(\mathcal{N}b) \
   \ \text{and} \ \ f^{-1}(y) = e h^{-1}(\mathcal{N} \cdot y/e^{2}). $$
This finishes the proof.
\end{proof}

Further, let $R$ be a Henselian valued ring with maximal ideal
$\mathfrak{m}$. Let $0 \leq r < n$, $p = (p_{r+1}, \ldots, p_{n})$
be an $(n-r)$-tuple of polynomials $p_{r+1},\ldots,p_{n} \in
R[X]$, $X = (X_{1},\ldots,X_{n})$,  and
$$ J := \frac{\partial(p_{r+1}, \ldots,
   p_{n})}{\partial(X_{r+1},\ldots,X_{n})}, \ \ e := J(\mathbf{0}). $$
Suppose that
$$ \mathbf{0} \in V := \{ x \in R^{n}: p_{r+1}(x) = \ldots = p_{n}(x) =
   0 \}. $$
In a similar fashion as above, we can establish the following
version of the implicit function theorem.

\begin{proposition}\label{implicit}
If $e \neq 0$, then there exists a continuous map
$$ \phi: (e^{2} \cdot \mathfrak{m})^{\times r} \to
   (e \cdot \mathfrak{m})^{\times (n-r)} $$
such that $\phi(0)=0$ and the graph map
$$ (e^{2} \cdot \mathfrak{m})^{\times r} \ni u \to (u,\phi(u)) \in
   (e^{2} \cdot \mathfrak{m})^{\times r} \times
   (e \cdot \mathfrak{m})^{\times (n-r)} $$
is an open embedding into the zero locus $V$ of the polynomials
$p$.
\end{proposition}

\begin{proof}
Put $f(X) := (X_{1},\ldots,X_{r},p(X))$; of course, the jacobian
determinant of $f$ at $\mathbf{0} \in R^{n}$ is equal to $e$. Keep
the notation from the proof of Proposition~\ref{H-3}, take any $b
\in e^{2} \cdot \mathfrak{m}^{\times r}$ and put $y := (e^{2}b,0)
\in R^{n}$. Then we have the equivalences
$$ f(eX)=y \ \Leftrightarrow \ f(eX) - y= \mathbf{0} \
   \Leftrightarrow \ e \mathcal{M}(\mathbf{0}) \cdot
   (X + \mathcal{N} g(X) - \mathcal{N} \cdot (b,0)) = \mathbf{0}. $$
Applying Corollary~\ref{H-2} to the map $h(X) := X +
\mathcal{N}g(X)$, we get
$$ f^{-1}(y) = ex \ \Leftrightarrow \ x = h^{-1}(\mathcal{N} \cdot (b,0)) \
   \ \text{and} \ \ f^{-1}(y) = e h^{-1}(\mathcal{N} \cdot y/e^{2}). $$
Therefore the function
$$ \phi(u) := e h^{-1}(\mathcal{N}\cdot (u,0)/e^{2}) $$
is the one we are looking for.
\end{proof}

The density property of Henselian non-trivially valued fields
follows immediately from Proposition~\ref{implicit} and the
Jacobian criterion for smoothness (see
e.g.~\cite[Theorem~16.19]{Eis}), recalled below for the reader's
convenience.

\begin{theorem}\label{smooth}
Let $I = (p_{1}, \ldots, p_{s}) \subset K[X]$, $X =
(X_{1},\ldots,X_{n})$ be an ideal, $A := K[X]/I$ and $V :=
\mathrm{Spec}\, (A)$. Suppose the origin $\mathbf{0} \in K^{n}$
lies in $V$ (equivalently, $I \subset (X)K[X]$) and $V$ is of
dimension $r$ at $\mathbf{0}$. Then the Jacobian matrix
$$ \mathcal{M} = \left[
   \frac{\partial p_{i}}{\partial X_{j}}(\mathbf{0}): \: i=1,\ldots,s, \: j=1,\ldots,n
   \right] $$
has rank $\leq (n-r)$ and $V$ is smooth at $\mathbf{0}$ iff
$\mathcal{M}$ has exactly rank $(n-r)$. Furthermore, if $V$ is
smooth at $\mathbf{0}$ and
$$ \det \left[
   \frac{\partial p_{i}}{\partial X_{j}}(\mathbf{0}): \: i,j=r+1,\ldots,n
   \right] \neq 0, $$
then $p_{r+1},\ldots,p_{n}$ generate the localization $I \cdot
K[X]_{(X_{1},\ldots,X_{n})}$ of the ideal $I$ with respect to the
maximal ideal $(X_{1},\ldots,X_{n})$.
\end{theorem}

Let us mention that we are currently preparing a series of papers
devoted to geometry of algebraic subsets of $K^{n}$, i.al.\ to the
results of our article~\cite{Now2}, for the case where the ground
field $K$ is an arbitrary Henselian valued field of
equicharacteristic zero. Finally, I wish to thank Laurent
Moret-Bailly for pointing out the implicit function theorem in the
paper~\cite{G-G-MB}.

\vspace{2ex}

\begin{small}
%\begin{sc}
Institute of Mathematics

Faculty of Mathematics and Computer Science

Jagiellonian University

%Faculty of Mathematics and Computer Science

ul.~Profesora \L{}ojasiewicza 6, 30-348 Krak\'{o}w, Poland

{\em e-mail address: nowak@im.uj.edu.pl}
%\end{sc}
\end{small}

\end{document}